\documentclass[twocolumn]{autart}    
\usepackage{graphicx,amsmath,amsfonts,amssymb,mathrsfs,bm,tabularx,array,enumerate,pifont,epsfig,epstopdf,color,tikz,subfigure,harvard}

\newtheorem{corollary}{Corollary}
\newtheorem{lemma}{Lemma}
\newtheorem{theorem}{Theorem}

\newtheorem{assumption}{Assumption}

\newtheorem{remark}{Remark}

\newenvironment{proof}{Proof.}{\hfill $\square$}
\begin{document}

\begin{frontmatter}

\title{Distributed sub-optimal resource allocation over weight-balanced graph via singular perturbation \thanksref{footnoteinfo}} 

\thanks[footnoteinfo]{This paper was not presented at any IFAC
meeting. Corresponding author Y.~Hong. Tel. +86-10-82541824.
Fax +86-10-82541832.}

\author[CAS]{Shu Liang}\ead{sliang@amss.ac.cn},    
\author[CAS]{Xianlin Zeng}\ead{xianlin.zeng@amss.ac.cn},               
\author[CAS]{Yiguang
Hong}\ead{yghong@iss.ac.cn}  

\address[CAS]{Key Laboratory of Systems and Control, Academy of
Mathematics and Systems Science, Chinese Academy of Sciences,
Beijing, 100190, China}                   


\begin{abstract}                          
In this paper, we consider distributed optimization design for resource allocation problems over weight-balanced graphs. With the help of singular perturbation analysis, we propose a simple sub-optimal continuous-time optimization algorithm. Moreover, we prove the existence and uniqueness of the algorithm equilibrium, and then show the convergence with an exponential rate. Finally, we verify the sub-optimality of the algorithm, which can approach the optimal solution as an adjustable parameter tends to zero. 
\end{abstract}

\end{frontmatter}

\section{Introduction}
Distributed optimization has attracted intense research attention in recent years, due to its theoretic significance and broad applications in various research fields, and many distributed algorithms have been developed to optimize a global objective or cost function based on agents' local cost functions and information exchange between neighbors in a multi-agent network \cite{Yuan2016Zeroth,Mokhtari2017Network}. So far, much effort has also been done for distributed continuous-time algorithm design, referring to \cite{Shi2013reaching,Gharesifard2014Distributed,Liu2015Second,Lou2016Distributed,Yang2017Multi} and the references therein, partially because of its applications in physical plants or hybrid systems and available continuous-time control methods.

Resource allocation is one of the most important optimization problems, which has been widely investigated in various areas such as economic systems, communication networks, and power grids; and various algorithms, centralized or decentralized have been constructed, for example, in \cite{Arrow1958Studies,Heal1969Planning,Lakshmanan2008Decentralized,Zappone2016Energy}.   Different from the most existing results, \cite{Cherukuri2016Initialization,Yi2016Initialization} considered distributed initialization-free continuous-time algorithms to solve the optimal resource allocation problem with applications to economic dispatch of power systems.  The algorithms given in \cite{Yi2016Initialization,Gharesifard2016Price} dealt with undirected graph cases, based on the symmetry of the Laplacians associated with the given graphs. As pointed out in \cite{Gharesifard2014Distributed,Gharesifard2016Price}, there were examples to make a distributed algorithm for undirected graphs divergent for some directed graphs. For practical applications, distributed optimization algorithms over balanced directed graphs were developed with or without the resource allocation constraint, for example, in \cite{Gharesifard2014Distributed,Cherukuri2016Initialization}.   However, these algorithms, involving the usage of the eigenvalues of the Laplacians, might yield additional computation burden in the distributed implementation, and make the convergence quite sensitive to the network topology.

Partially because distributed optimization just became a hot topic in this decade, there are quite few results about its sub-optimal algorithms and related analysis. For example, \cite{Nedic2009Approximate} proposed an algorithm without exactly solving the considered problem, but with fast convergence rate.  In fact, sub-optimal design deserves investigation, though the exactness of optimal solutions may be sacrificed.   As we know, the exact optimization solution may be hard to obtain due to technical difficulties, complexity, or computational cost; on the contrary, sub-optimal algorithms may provide considerable benefits with simple feasible designs and even performance enhancement.  In distributed design for large-scale networks, we may particularly need sub-optimal algorithms to reduce the computational complexity or sensitivity to the network topology, rather than to seek high-cost exact optimal solution \cite{Bhatti2016Large}.

Based on the above observation, the motivation of this paper is to study a distributed sub-optimal algorithm design for the resource allocation optimization over a balanced directed graph. Our algorithm is of lower dimensions than existing ones, with the reduction of computational burden and information exchanging. Moreover, its convergence is kept over any strongly connected and weight-balanced graph because its design does not depend on any specific knowledge of the graph.   To achieve this, we adopt a singular perturbation idea in the distributed sub-optimal design. Note that the singular perturbation theory provides powerful tools for (continuous-time) control design \cite{Kokotovic1999Singular}, and the well-known high-gain technique and semi-global stabilization design are closely related to singular perturbation \cite{Khalil2002Nonlinear}.

The contributions of this paper can be summarized as follows. (i) We first propose a distributed sub-optimal algorithm to solve the continuous-time resource allocation problem for weight-balanced graphs, without using any information of the network topology. The sub-optimal design is simpler than those optimization ones. In light of the conventional fixed-point theory, we prove the existence and uniqueness of the algorithm equilibrium. (ii)   We adopt a singular perturbation idea in our design, totally different from that given in \cite{Gharesifard2014Distributed,Cherukuri2016Initialization,Yi2016Initialization}, and then show that the quasi-steady-state model of our algorithm is exactly the primal-dual optimization algorithm. Note that the original primal-dual algorithm may not be directly implementable in a fully distributed manner due to the coupled resource allocation constraint. (iii) We prove the convergence of the proposed sub-optimal algorithm with an exponential rate, and estimate the difference of the sub-optimal solution from the optimal one, which, in fact, is bounded linearly by an adjustable parameter.  Moreover, we verify that the sub-optimal solution always satisfies the resource allocation constraint and can be made arbitrarily close to the optimal point as the parameter tends to $0$.

The paper organization is as follows: Section 2 provides preliminaries and formulates the problem, while Section 3 proposes the distributed algorithms. Then Section 4 presents the algorithm analysis, and finally, Section 5 gives some concluding remarks.

{\em Notations: } Let $\mathbb{R}^n$ be the $n$-dimensional real vector space and $\mathbb{B}$ be the unit ball. The Euclidean norm of vectors in $\mathbb{R}^n$ and its induced consistent matrix norm are denoted by $\|\cdot\|$. $col(x_1,...,x_N)$ stands for the column vector stacked with column vectors $x_i,\,(i=1,...,N)$, i.e., $col(x_1,...,x_N) = (x_1^{T},\,x_2^{T},\,\cdots,\,x_n^{T})^{T}$, and $1_{n} = col\{1,...,1\}\in \mathbb{R}^n$. $I_n$ is the identity matrix in $\mathbb{R}^{n\times n}$.  $\otimes$ denotes the Kronecker's product for matrices and $\det(\cdot)$ denotes the determinant of a matrix. For a smooth function $f:\mathbb{R}^n\to \mathbb{R}$, $\nabla f(x)$ and $\nabla^2 f(x)$ denote its gradient vector and Hessian matrix at point $x$, respectively.

\section{Preliminaries and Formulation}

In this section, we introduce relevant preliminary knowledge about convex analysis and graph theory and then formulate our problem.

\subsection{Preliminaries}
A function $f: \mathbb{R}^n\to \mathbb{R}$ is said to be {\em convex} if $f(\lambda z_1
+(1-\lambda)z_2) \leq \lambda f(z_1) + (1-\lambda)f(z_2)$ for any $z_1, z_2 \in \mathbb{R}^n$ and $\lambda\in (0,\,1)$. Moreover, it is said to be {\em $c_0$-strongly convex} for a constant $c_0>0$, if
\begin{multline}
f(\lambda z_1 +(1-\lambda)z_2) \leq \lambda f(z_1) + (1-\lambda)f(z_2) \\
- \frac{1}{2}c_0\lambda(1-\lambda)\|z_1-z_2\|^2.
\end{multline}
For a twice continuously differentiable function $f$, it is $c_0$-strongly convex if and only if $\nabla^2 f(x) \geq c_0 I_n$. In addition, for $c_0$-strongly convex and differentiable function $f$, there holds
\begin{equation}\label{eq:stronglyConvex}
f(y) \geq f(x) + \nabla f(x)^T (y-x) + \frac{1}{2}c_0\|y-x\|^2, \,\forall\, x,y\in \mathbb{R}^n.
\end{equation}

A function $g:\mathbb{R}^n\to \mathbb{R}$ is said to be {\em level bounded} \cite{Rockafellar1998Variational} if all sets of the form
\begin{equation}
\{x\in \mathbb{R}^n\,|\, g(x)\leq \alpha\}, \text{ for } \alpha\in \mathbb{R}^n
\end{equation}
are bounded. Obviously, the strong convexity and differentiability imply the level boundedness by \eqref{eq:stronglyConvex}.

A map $H:\mathbb{R}^n\to \mathbb{R}^n$ is said to be {\em locally Lipschitz continuous} at a point $x$ if there are constants $\delta>0$ and $\kappa = \kappa(x, \delta)$ such that
\begin{equation}\label{eq:LipchitzContinuous0}
\|H(x_1) - H(x_2)\| \leq \kappa \|x_1 - x_2\|, \, \forall\, x_1, x_2 \in x + \delta\mathbb{B}.
\end{equation}
Moreover, $H$ is said to be {\em $\kappa$-Lipshcitz continuous} if \eqref{eq:LipchitzContinuous0} holds irrespective of $x$ and $\delta$.

Consider a multi-agent network with its interaction topology described by a weighted graph $\mathcal{G} = \{\mathcal{V}, \mathcal{E}, \mathcal{A}\}$, where $\mathcal{V}= \{ 1,2, \ldots N\}$ is the node set, $\mathcal{E} \subseteq \mathcal{V} \times \mathcal{V}$ is the edge set, and $\mathcal{A} = [a_{ij}]_{N\times N}$ is an adjacency matrix with $a_{ij} >0$ if $(j,i) \in \mathcal{E}$ (meaning that agent $j$ can send its information to agent $i$), and $a_{ij} =0$, otherwise. If $a_{ij} = a_{ji}, \, \forall\, i,j \in \mathcal{V}$, then $\mathcal{G}$ is undirected. A path is a sequence of vertices connected by edges. A graph is said to be strongly connected if there is a path between any pair of vertices. For node $i \in \mathcal{V}$, the weighted in-degree and weighted out-degree are $d_{in}^i =\sum_{j=1}^N a_{ij}$ and $d_{out}^i =\sum_{j=1}^N a_{ji}$, respectively. A graph is weight-balanced if $\forall\,i \in \mathcal{V}, d_{in}^i = d_{out}^i$. The following lemma characterizes graph $\mathcal{G}$ by its (in-degree) Laplacian matrix, defined as $L= \mathcal{D}_{in} - \mathcal{A}$, where $\mathcal{D}_{in} = diag\{d_{in}^1, \ldots, d_{in}^N\} \in \mathbb{R}^{N \times N}$.

\begin{lemma}\cite{Bullo2009Distributed}
The following statements hold.
\begin{enumerate}[1)]
\item Graph $\mathcal{G}$ is undirected if and only if $L = L^T$.
\item Graph $\mathcal{G}$ is strongly connected if and only if zero is a simple eigenvalue of $L$.
\item Graph $\mathcal{G}$ is weight-balanced if and only if $L + L^T$ is positive semidefinite.
\end{enumerate}
\end{lemma}

\subsection{Problem formulation}

Distributed resource allocation optimization problem is usually formulated as follows. For each agent $i \in \mathcal{V}$, there are a local decision variable $x_i\in \mathbb{R}^{n}$ and a local cost function $f_i(x_i):\mathbb{R}^n\to \mathbb{R}$. The agents cooperate each other in order to minimize the total cost function of the network, defined as $f(\bm{x}) \triangleq \sum_{i=1}^Nf_i(x_i)$, subject to the resource allocation constraint $\sum_{i=1}^Nx_i = \sum_{i=1}^Nb_i=d$. In other words,
\begin{equation}\label{eq:optimizationProblem}
\min_{\bm{x}\in \mathbb{R}^{nN}} f(\bm{x}), \text{ s.t. } (1_N^T\otimes I_n)\bm{x} = d,
\end{equation}
where $\bm{x} \triangleq col\{x_1,...,x_N\}$ and $d\in \mathbb{R}^n$.

The following assumption is adopted to ensure the well-posedness of \eqref{eq:optimizationProblem}, which is widely used.

\begin{assumption}\label{assum:1}
~
\begin{enumerate}[1)]
\item $f(\bm{x})$ is $c_0$-strongly convex and twice continuously differentiable.
\item The interaction graph $\mathcal{G}$ is strongly connected and weight-balanced.
\end{enumerate}
\end{assumption}

The following lemma is quite fundamental for problem \eqref{eq:optimizationProblem}. We present it with its proof here for completeness.

\begin{lemma}\label{lem:optimalPrimalDual}
Under Assumption \ref{assum:1}, there exists a unique optimal solution $\bm{x}^* = col\{x_1^*, ..., x_N^*\}$ of problem \eqref{eq:optimizationProblem}. In addition, there exists a unique $\bm{\lambda}^* = col\{\mu^*, ..., \mu^*\}$ such that the following condition holds.
\begin{equation}\label{eq:optimalSolution}
\left\{\begin{aligned}
0 & = \nabla f(\bm{x}^*) + \bm{\lambda}^*\\
0 & = (1_N^T\otimes I_n)\bm{x}^* -d
\end{aligned}\right.
\end{equation}
\end{lemma}

\begin{proof}
Since $f$ is strongly convex and differentiable, it is level bounded, which implies the existence of an optimal point over the set $\Omega = \{\bm{x}\in \mathbb{R}^{nN}\,|\,(1_N^T\otimes I_n)\bm{x} -d = 0\}$. Also, the strong convexity of $f$ implies the uniqueness of the optimal point $\bm{x}^*$.  Since the normal cone of $\Omega$ at point $\bm{x}^*$ is $\mathcal{N}_{\Omega}(\bm{x}^*) = \{1_N\otimes \mu\,|\,\mu\in\mathbb{R}^n\}$, the conclusion follows from the necessary optimality condition $-\nabla f(\bm{x}^*) \in \mathcal{N}_{\Omega}(\bm{x}^*)$ \cite[Theorem 6.12, page 207]{Rockafellar1998Variational}.
\end{proof}

The goal of this paper is to design a distributed sub-optimal algorithm with a positive adjustable parameter $\varepsilon$ for problem \eqref{eq:optimizationProblem}, such that
\begin{enumerate}[1)]
\item the equilibrium point of the proposed algorithm is exponentially stable with the resource allocation constraint held;
\item it approaches the optimal solution of problem \eqref{eq:optimizationProblem} as $\varepsilon \to 0$, and the difference between it and the optimal solution is bounded linearly by $\varepsilon$.
\end{enumerate}
Of course, the design of sub-optimal algorithm should be simpler than that for optimization algorithms.

\section{Distributed algorithm design}

In this section, we propose a distributed sub-optimal algorithm, and also show the relationship between its design and singular perturbation analysis.

To make a comparison, we first introduce a distributed algorithm over undirected graphs for problem \eqref{eq:optimizationProblem}, obtained in the literature, such as \cite{Yi2016Initialization}:
\begin{equation}\label{eq:algorithmPI}
\forall\, i\in \mathcal{V}, \, \left\{\begin{aligned}
\dot{x}_i & = - \nabla f_i(x_i) - \lambda_i\\
\dot{\lambda}_i & = - k_P\sum_{j=1}^Na_{ij}(\lambda_i-\lambda_j) \\
&\quad - k_I\sum_{j=1}^Na_{ij}(z_i-z_j) + x_i - b_i\\
\dot{z}_i & = \sum_{j=1}^Na_{ij}(\lambda_i-\lambda_j)
\end{aligned}\right.
\end{equation}
where $\sum_{i=1}^N b_i= d$ and $k_P=k_I=1$ in \cite{Yi2016Initialization}. The continuous-time algorithm \eqref{eq:algorithmPI} is constructed by combining the Lagrangian duality and the consensus dynamics. Roughly speaking, the dynamics of $x_i$'s correspond to the gradient decent and the dynamics of $\lambda_i$'s and $z_i$'s render the local Lagrangian multipliers $\lambda_i$ to reach a consensus at the optimal point of the dual problem.

On the other hand, as pointed out in \cite{Gharesifard2014Distributed,Gharesifard2016Price}, the continuous-time algorithms like \eqref{eq:algorithmPI} may become divergent over some directed graphs. One remedy is to tune the parameters $k_P$ and $k_I$ to stabilize the algorithm dynamics over a balanced graph, which was indeed used in \cite{Gharesifard2014Distributed,Cherukuri2016Initialization}. However, since that stabilization is based on the eigenvalues of the Laplacian of the balanced graph, whose information is not local, the algorithm is not fully distributed or its design increases the computational cost.

In this paper, we propose a simple distributed algorithm for problem \eqref{eq:optimizationProblem} without the knowledge of the eigenvalues associated with the considered balanced graph:
\begin{equation}\label{eq:algorithmNew}
\forall\, i\in \mathcal{V}, \, \left\{\begin{aligned}
\dot{x}_i & = - \nabla f_i(x_i) - \lambda_i\\
\varepsilon\dot{\lambda}_i & = - \sum_{j=1}^Na_{ij}(\lambda_i-\lambda_j) + \varepsilon (x_i - b_i)
\end{aligned}\right.
\end{equation}
where $\varepsilon > 0$ is a small adjustable parameter. For simplicity, we rewrite algorithm \eqref{eq:algorithmNew} in a compact form as
\begin{equation}\label{eq:algorithmCompact}
\left\{\begin{aligned}
\dot{\bm{x}} & = -\nabla f(\bm{x}) - \bm{\lambda}\\
\varepsilon\dot{\bm{\lambda}} & = - \bm{L}\bm{\lambda} + \varepsilon(\bm{x} - \bm{b})
\end{aligned}\right.
\end{equation}
where $\bm{\lambda} = col\{\lambda_1,...,\lambda_N\}, \bm{b} = col\{b_1,...,b_N\}$ and $\bm{L} = L\otimes I_n$, $L$ is the Laplacian matrix of the strongly connected and weight-balanced graph.

\begin{remark}
Algorithm \eqref{eq:algorithmNew} has lower dimensions and less (communication) complexity than \eqref{eq:algorithmPI}, because it does not involve the dynamics of $z_i$'s and related information exchanging.
\end{remark}

Since $\bm{L}$ is generally asymmetric, \eqref{eq:algorithmCompact} loses any interpretation from gradient-decent-gradient-ascent dynamics for the saddle-point computation, which is widely used for constrained convex optimization. In fact, our design is based on singular perturbation ideas as follows.
Clearly, we can choose a matrix $T\in \mathbb{R}^{N\times N}$ satisfying
\begin{equation}
T = [1_N, M_1]^T, \quad T^{-1} = [1_N, M_2].
\end{equation}
Let $[\begin{smallmatrix}
\mu\\
\bm{\theta}
\end{smallmatrix}] \triangleq (T\otimes I_n) \bm{\lambda}$, where $\mu \in \mathbb{R}^{n}, \bm{\theta} \in \mathbb{R}^{n(N-1)}$. Then \eqref{eq:algorithmCompact} can be written as a standard singular perturbation model as follows:
\begin{equation}\label{eq:singularPerturbation}
\left\{\begin{aligned}
\dot{\bm{x}} & = -\nabla f(\bm{x}) - (1_N\otimes I_n)\mu - (M_2\otimes I_n)\bm{\theta}\\
\dot{\mu} & = (1_N^T\otimes I_n)\bm{x} - d\\
\varepsilon\dot{\bm{\theta}} & = - (M_1^TLM_2\otimes I_n) \bm{\theta} + \varepsilon (M_1^T\otimes I_n) (\bm{x} - \bm{b})
\end{aligned}\right.
\end{equation}
It can be observed from \eqref{eq:singularPerturbation} that, for a sufficiently small $\varepsilon >0$, $\bm{\theta}$ corresponds to the fast transient part and $(\bm{x}, \mu)$ corresponds to the slow part. Because all the eigenvalues of matrix $- (M_1^TLM_2\otimes I_n)$ are negative, the fast manifold is simply $\bm{\theta} = 0$, and then the quasi-steady-state model (or reduced model) of \eqref{eq:singularPerturbation} is
\begin{equation}\label{eq:quasiSteady}
\left\{\begin{aligned}
\dot{\bm{x}} & = -\nabla f(\bm{x}) - (1_N\otimes I_n)\mu\\
\dot{\mu} & = (1_N^T\otimes I_n)\bm{x} - d\\
\end{aligned}\right.
\end{equation}
Let us denote the solution of \eqref{eq:quasiSteady} by $(\tilde{\bm{x}}(t), \tilde{\mu}(t))$ and the solution of \eqref{eq:singularPerturbation} by $(\bm{x}(t,\varepsilon), \mu(t,\varepsilon), \bm{\theta}(t,\varepsilon))$.  With the existing singular perturbation results \cite[Theorems 11.2 and 11.3, pages 439 and 452]{Khalil2002Nonlinear},  \eqref{eq:singularPerturbation} is asymptotically stable and, for any $\varepsilon \in (0, \varepsilon^*)$ with some $\varepsilon^*>0$, an initial moment $t_0$ and some time $t_b>t_0$, we have
\begin{equation}\label{eq:singularPerturbationResult}
\begin{aligned}
(\bm{x}(t,\varepsilon), \mu(t,\varepsilon))- (\tilde{\bm{x}}(t),\tilde{\mu}(t)) & = O(\varepsilon), \, t\in [t_0,\infty)\\
\bm{\theta}(t,\varepsilon) - 0 & = O(\varepsilon), \, t\in [t_b,\infty)
\end{aligned}
\end{equation}

To sum up, we have the following statements from singular perturbation analysis.
\begin{enumerate}[1)]
\item The algorithm \eqref{eq:algorithmCompact} has its quasi-steady-state model as \eqref{eq:quasiSteady}, and \eqref{eq:quasiSteady} is exactly the primal-dual optimization algorithm for problem \eqref{eq:optimizationProblem}. However, in contrast to \eqref{eq:algorithmCompact}, the algorithm \eqref{eq:quasiSteady} is not directly implementable in a fully distributed manner because the dynamics of $\mu$ needs to collect all the information of $x_1, ..., x_N$ due to the coupled resource allocation constraint.

\item  The trajectory of algorithm \eqref{eq:algorithmCompact} is near the (centralized) primal-dual one within an error bound estimation $O(\varepsilon)$. Moreover, since $(\tilde{\bm{x}}(t),\tilde{\mu}(t))$ converges to the optimal primal-dual solution $(\bm{x}^*, \mu^*)$ in Lemma \ref{lem:optimalPrimalDual}, \eqref{eq:algorithmCompact} approaches a ball centered at the optimal point as $t\to \infty$, yielding some sub-optimal solution.
\end{enumerate}
Note that for our algorithm, we have to verify the existence of its equilibrium, which is not straightforward.  Moreover, the estimation $O(\varepsilon)$ in the singular perturbation theory may be too rough since it holds for all $t\in [t_0,\infty)$. In order to clarify the effectiveness of our method, we have to find a new way for the algorithm analysis. To be specific, we will first study the existence of the equilibrium of the algorithm \eqref{eq:algorithmCompact}, and then study its convergence and sub-optimality, in the sequel.

\section{Main Results}
In this section, we analyze the equilibrium, convergence and sub-optimality for the algorithm \eqref{eq:algorithmCompact}.

\subsection{Equilibrium analysis}
Here let us show the existence and uniqueness of the equilibrium of algorithm \eqref{eq:algorithmCompact}.

\begin{theorem}\label{thm:1}
Under Assumption \ref{assum:1}, there exists $\varepsilon_0>0$ such that for any fixed $\varepsilon \in (0, \, \varepsilon_0)$, algorithm \eqref{eq:algorithmCompact} has a unique equilibrium, i.e., a unique pair $(\bar{\bm{x}}(\varepsilon), \bar{\bm{\lambda}}(\varepsilon))$ satisfying the following equation
\begin{equation}\label{eq:equilibriumCompact}
\left\{\begin{aligned}
0 &= \nabla f(\bm{x}) + \bm{\lambda}\\
0 &= -\varepsilon (\bm{x} - \bm{b}) + \bm{L}\bm{\lambda}
\end{aligned}\right.
\end{equation}
\end{theorem}

\begin{proof}  We first show the existence and then the uniqueness in the proof.

(i) Existence: Let $(\bm{x}^*, \bm{\lambda}^*)$ be the optimal solution pair in \eqref{eq:optimalSolution}. Since $f(\bm{x})$ is twice continuously differentiable,
\begin{equation}\label{eq:rz}
\nabla f(\bm{x}) = \nabla f(\bm{x}^*) + H \bm{z} - \bm{r}(\bm{z}),
\end{equation}
where
\begin{equation}
H \triangleq \nabla^2f(\bm{x}^*), \quad \bm{z} \triangleq \bm{x} - \bm{x}^*,
\end{equation}
and $\bm{r}(\bm{z})$ is an infinitesimal term with respect to $\bm{z}$.

Clearly, it follows from \eqref{eq:optimalSolution} that $\bm{L}\bm{\lambda}^* = 0$ and $\bm{L}\nabla f(\bm{x}^*) = -\bm{L}\bm{\lambda}^* = 0$. Moreover, since $(1_N^T\otimes I_n)(\bm{b} - \bm{x}^*) = 0$, there exists $\bm{\lambda}_0$ such that $\bm{L}\bm{\lambda}_0 = \bm{b} - \bm{x}^*$. Thus, by eliminating $\bm{\lambda}$ in \eqref{eq:equilibriumCompact}, we obtain an equation with respect to variable $\bm{z}$ as
\begin{equation}\label{eq:Phi}
\begin{aligned}
\bm{z} = \Phi(\bm{z}, \varepsilon) & \triangleq (\varepsilon I_{nN} + \bm{L}H)^{-1}(\varepsilon(\bm{b} - \bm{x}^*) + \bm{L}\bm{r}(\bm{z}))\\
& = (\varepsilon I_{nN} + \bm{L}H)^{-1}\bm{L}(\varepsilon\bm{\lambda}_0 + \bm{r}(\bm{z}))
\end{aligned}
\end{equation}
We claim that matrix $\varepsilon I_{nN} + \bm{L}H$ is nonsingular (and then the map $\Phi(\bm{z},\varepsilon)$ in \eqref{eq:Phi} is well-defined). In fact, $H \geq c_0I_{nN}$ and $\bm{L}+ \bm{L}^T$ is positive semidefinite according to Assumption \ref{assum:1}. Then $v^TH^{\frac{1}{2}}\bm{L}H^{\frac{1}{2}}v = v^TH^{\frac{1}{2}}(\bm{L}+\bm{L}^T)H^{\frac{1}{2}}v \geq 0, \,\forall\, v \in \mathbb{R}^{nN}$.  Due to $\det(sI_{nN} - \bm{L}H) = \det(sI_{nN} - H^{\frac{1}{2}}\bm{L}H^{\frac{1}{2}})$, all the eigenvalues of matrix $\bm{L}H$ are nonnegative. Consequently, matrix $\varepsilon I_{nN} + \bm{L}H$ is nonsingular.

Moreover, since $ (\varepsilon I_{nN} + \bm{L}H)^{-1}(\varepsilon I_{nN} + \bm{L}H) = I_{nN}$,
\begin{equation}
\begin{aligned}
(\varepsilon I_{nN} + \bm{L}H)^{-1}\bm{L} & = H^{-1} -  (H + \varepsilon^{-1}H\bm{L}H)^{-1}.
\end{aligned}
\end{equation}
Note that $\|\nu\|^2 = \eta\nu^T(H + \varepsilon^{-1}H\bm{L}H)\nu \geq \eta c_{0} \|\nu\|^2$ for any eigenvalue $\eta$ of matrix $(H + \varepsilon^{-1}H\bm{L}H)^{-1}$ with corresponding eigenvector $\nu \neq 0$. Hence, the spectral radius $\rho$ of matrix $(H + \varepsilon^{-1}H\bm{L}H)^{-1}$ satisfies
\begin{equation}
\rho((H + \varepsilon^{-1}H\bm{L}H)^{-1})\leq c_0^{-1}, \, \forall\, \varepsilon >0.
\end{equation}
Then, recalling \cite[Lemma 5.6.10, page 347]{Horn2013Matrix}, there exists a matrix norm $\|\cdot\|_{\sharp}$ such that
\begin{equation}
\|(H + \varepsilon^{-1}H\bm{L}H)^{-1}\|_{\sharp} \leq c_0^{-1} +1, \, \forall\, \varepsilon >0.
\end{equation}
It follows from the equivalence of matrix norms that there exists a constant $k_0>0$ such that
\begin{equation}
\|(H + \varepsilon^{-1}H\bm{L}H)^{-1}\| \leq k_0(c_0^{-1} +1), \, \forall\, \varepsilon >0.
\end{equation}
Therefore, for any $\varepsilon >0$,
\begin{equation}\label{eq:k1}
\begin{aligned}
\|(\varepsilon I_{nN} + \bm{L}H)^{-1}\bm{L}\| & \leq \|H^{-1}\| + \|(H + \varepsilon^{-1}H\bm{L}H)^{-1}\|\\
& \leq (k_0+1)c_0^{-1} + k_0 \triangleq k_1.
\end{aligned}
\end{equation}
Additionally, for $\bm{r}(\bm{z})$ in \eqref{eq:rz} and $k_1$ in \eqref{eq:k1}, there exists $\delta = \delta(k_1)>0$ such that
\begin{equation}\label{eq:rz2}
\|\bm{r}(\bm{z}) - \bm{r}(\bm{z}')\|\leq \frac{1}{k_1+1}\|\bm{z} - \bm{z}'\|, \, \forall\, \bm{z},\bm{z}' \in \delta\mathbb{B}.
\end{equation}
Furthermore, for the constants $k_1 >0, \delta>0$ and $\bm{\lambda}_0$ in \eqref{eq:Phi}, there exists $\varepsilon_0  = \varepsilon_0(k_1, \delta, \bm{\lambda}_0)>0$ such that
\begin{equation}\label{eq:epsilone0}
\varepsilon \|\bm{\lambda}_0\| \leq \frac{\delta}{k_1(k_1 +1)}, \,\forall\, \varepsilon \in (0,\,\varepsilon_0).
\end{equation}

Consider the map $\Phi(\bm{z}, \varepsilon)$ in \eqref{eq:Phi}. On the one hand, it follows from \eqref{eq:k1} and \eqref{eq:rz2} that
\begin{equation}\label{eq:contraction}
\|\Phi(\bm{z}, \varepsilon) - \Phi(\bm{z}', \varepsilon)\|\leq \frac{k_1}{k_1+1}\|\bm{z}-\bm{z}'\|, \, \forall\, \bm{z},\bm{z}' \in \delta\mathbb{B}
\end{equation}
for any fixed $\varepsilon >0$, that is, $\Phi(\cdot,\varepsilon)$ is a contraction map in $\delta\mathbb{B}$. On the other hand, it follows from \eqref{eq:rz2} and \eqref{eq:epsilone0} that
\begin{equation}\label{eq:maptoitself}
\|\Phi(\bm{z},\varepsilon)\|\leq k_1\varepsilon \|\bm{\lambda}_0\| + k_1\|\bm{r}(\bm{z})\|\leq \delta, \,\forall\, \bm{z}\in \delta\mathbb{B}
\end{equation}
for any $\varepsilon \in (0,\,\varepsilon_0)$, that is, $\Phi(\cdot,\varepsilon)$ maps the compact set $\delta\mathbb{B}$ into itself.   According to the Contraction Mapping Theorem \cite[page 458]{Bertsekas2015Convex}, $\Phi(\cdot,\varepsilon)$ has a fixed point $\bar{\bm{z}}(\varepsilon)$, which is the solution of equation \eqref{eq:Phi}. Let
\begin{equation}
\bar{\bm{x}}(\varepsilon) \triangleq \bar{\bm{z}}(\varepsilon) + \bm{x}^*, \quad \bar{\bm{\lambda}}(\varepsilon) \triangleq - \nabla f(\bar{\bm{x}}(\varepsilon)).
\end{equation}
Thus, we obtain that $(\bar{\bm{x}}(\varepsilon), \bar{\bm{\lambda}}(\varepsilon))$ is a solution of equation \eqref{eq:equilibriumCompact}.

(ii) Uniqueness: Suppose there are two solution pairs $(\bm{x}, \bm{\lambda})$ and $(\bm{x}', \bm{\lambda}')$ for equation \eqref{eq:equilibriumCompact}. By some calculations, we have
\begin{equation}
\begin{aligned}
0 & = (\bm{x} - \bm{x}')^T(\nabla f(\bm{x}) - \nabla f(\bm{x}') + \bm{\lambda} - \bm{\lambda}') \\
&\quad + (\bm{\lambda} - \bm{\lambda}')^T(-(\bm{x} - \bm{x}') + \varepsilon^{-1}\bm{L}(\bm{\lambda} - \bm{\lambda}'))\\
& = (\bm{x} - \bm{x}')^T(\nabla f(\bm{x}) - \nabla f(\bm{x}')) \\
& \quad + \varepsilon^{-1} (\bm{\lambda} - \bm{\lambda}')^T\bm{L}(\bm{\lambda} - \bm{\lambda}') \geq 0
\end{aligned}
\end{equation}
Then $(\bm{x} - \bm{x}')^T(\nabla f(\bm{x}) - \nabla f(\bm{x}')) = 0$.  Since $f(\bm{x})$ is strongly convex, there must hold $\bm{x}' = \bm{x}$ and $\bm{\lambda}' = \bm{\lambda} = - \nabla f(\bm{x})$, which completes the proof.
\end{proof}

Note that the equilibrium is not known beforehand and the existing singular perturbation techniques do not cover this problem. Instead, we use a fixed-point theorem to prove the existence and then the uniqueness.

\subsection{Convergence and sub-optimality}

Based on the existence of the equilibrium, it is time to study the convergence of the proposed algorithm.

\begin{theorem}\label{thm:2}
Under Assumption \ref{assum:1}, the algorithm \eqref{eq:algorithmCompact} with $\varepsilon \in (0, \, \varepsilon_0)$ converges to its equilibrium point $(\bar{\bm{x}}(\varepsilon), \bar{\bm{\lambda}}(\varepsilon))$. Furthermore, if the gradient map $\nabla f(\bm{x})$ is $\kappa$-Lipshcitz continuous for some constant $\kappa >0$, then \eqref{eq:algorithmCompact} exponentially converges to its equilibrium point.
\end{theorem}

\begin{proof}
Since the righthand side of \eqref{eq:algorithmCompact} is locally Lipschitz continuous, there exists a unique trajectory $(\bm{x}(t,\varepsilon), \bm{\lambda}(t, \varepsilon))$ satisfying \eqref{eq:algorithmCompact}.   Take the following Lyapunov function
\begin{equation}
V(\bm{x}, \bm{\lambda}) \triangleq \|\bm{x} - \bar{\bm{x}}(\varepsilon)\|^2 + \|\bm{\lambda} - \bar{\bm{\lambda}}(\varepsilon)\|^2.
\end{equation}
Then $V$ is positive definite and its first order derivative with respect to time $t$ is
\begin{equation}
\begin{aligned}
\dot{V}(\bm{x}, \bm{\lambda}) & = - (\bm{x} - \bar{\bm{x}}(\varepsilon))^T (\nabla f(\bm{x}) + \bm{\lambda}) \\
& \quad - (\bm{\lambda} - \bar{\bm{\lambda}}(\varepsilon))^T (\bm{b}-\bm{x} + \varepsilon^{-1}\bm{L}\bm{\lambda})\\
& = - (\bm{x} - \bar{\bm{x}} (\varepsilon))^T (\nabla f(\bm{x}) - \nabla f(\bar{\bm{x}} (\varepsilon)))\\
& \quad - \varepsilon^{-1}(\bm{\lambda} - \bar{\bm{\lambda}}(\varepsilon))^T\bm{L}(\bm{\lambda} - \bar{\bm{\lambda}}(\varepsilon))\\
& \leq - c_0 \|\bm{x} - \bar{\bm{x}} (\varepsilon)\|^2 \\
& \quad - \varepsilon^{-1}(\bm{\lambda} - \bar{\bm{\lambda}}(\varepsilon))^T(\bm{L} + \bm{L}^T)(\bm{\lambda} - \bar{\bm{\lambda}}(\varepsilon))\\
& \leq 0.
\end{aligned}
\end{equation}
Also, we have that $\dot{V}(\bm{x}, \bm{\lambda}) = 0$ if and only if $\bm{x} = \bar{\bm{x}}(\varepsilon)$ and $\bm{\lambda} = \bar{\bm{\lambda}}(\varepsilon)$.  By the Invariance Principle \cite[page 126]{Khalil2002Nonlinear}, algorithm \eqref{eq:algorithmCompact} converges to $(\bar{\bm{x}}(\varepsilon), \bar{\bm{\lambda}}(\varepsilon))$.

Moreover, a linearized system of algorithm \eqref{eq:algorithmCompact} at its equilibrium can be obtained via replacing the term $\nabla f(\bm{x})$ by an affine map $F(\bm{x})$ defined as
\begin{equation}\label{eq:Fx}
F(\bm{x}) \triangleq \nabla f(\bar{\bm{x}}(\varepsilon)) + \nabla^2 f(\bar{\bm{x}}(\varepsilon))(\bm{x} - \bar{\bm{x}}(\varepsilon)).
\end{equation}
Following the same proof as above, this linear system is asymptotically stable.
%
Then there exist two positive definite matrices $P, Q \in \mathbb{R}^{2nN \times 2nN}$ such that
\begin{equation}
\begin{aligned}
V_1(\bm{x}, \bm{\lambda}) & \triangleq \begin{bmatrix}
\bm{x} - \bar{\bm{x}}(\varepsilon)\\
\bm{\lambda} - \bar{\bm{\lambda}}(\varepsilon)
\end{bmatrix}^T P \begin{bmatrix}
\bm{x} - \bar{\bm{x}}(\varepsilon)\\
\bm{\lambda} - \bar{\bm{\lambda}}(\varepsilon)
\end{bmatrix}\\
& \geq \zeta_1(\|\bm{x} - \bar{\bm{x}}(\varepsilon)\|^2 + \|\bm{\lambda} - \bar{\bm{\lambda}}(\varepsilon)\|^2)
\end{aligned}
\end{equation}
for some $\zeta_1>0$ and
\begin{equation}
\begin{aligned}
\dot{V}_1(\bm{x}, \bm{\lambda}) & = - \begin{bmatrix}
\bm{x} - \bar{\bm{x}}(\varepsilon)\\
\bm{\lambda} - \bar{\bm{\lambda}}(\varepsilon)
\end{bmatrix}^T Q \begin{bmatrix}
\bm{x} - \bar{\bm{x}}(\varepsilon)\\
\bm{\lambda} - \bar{\bm{\lambda}}(\varepsilon)
\end{bmatrix}\\
& \quad - 2 \begin{bmatrix}
\bm{x} - \bar{\bm{x}}(\varepsilon)\\
\bm{\lambda} - \bar{\bm{\lambda}}(\varepsilon)
\end{bmatrix}^T P \begin{bmatrix}
\nabla f(\bm{x}) - F(\bm{x})\\
0
\end{bmatrix}
\end{aligned}
\end{equation}
where $F(\bm{x})$ is in \eqref{eq:Fx}.

If $\nabla f(\bm{x})$ is $\kappa$-Lipschitz continuous, then there are positive constants $\zeta_2$ and $\zeta_3>0$ such that
\begin{equation}
\dot{V}_1(\bm{x}, \bm{\lambda}) \leq \zeta_2\|\bm{x} - \bar{\bm{x}}(\varepsilon)\|^2 - \zeta_3\|\bm{\lambda} - \bar{\bm{\lambda}}(\varepsilon)\|^2.
\end{equation}
Define a new Lyapunov function as
\begin{equation}
V_2(\bm{x}, \bm{\lambda}) \triangleq c_0^{-1}(\zeta_2 + \zeta_3)V(\bm{x}, \bm{\lambda}) + V_1(\bm{x}, \bm{\lambda})
\end{equation}
Then
\begin{equation}
V_2(\bm{x}, \bm{\lambda}) \geq (c_0^{-1}(\zeta_2 + \zeta_3) + \zeta_1)(\|\bm{x} - \bar{\bm{x}}(\varepsilon)\|^2 + \|\bm{\lambda} - \bar{\bm{\lambda}}(\varepsilon)\|^2),
\end{equation}
and
\begin{equation}
\dot{V}_2(\bm{x}, \bm{\lambda}) \leq - \zeta_3 (\|\bm{x} - \bar{\bm{x}}(\varepsilon)\|^2 + \|\bm{\lambda} - \bar{\bm{\lambda}}(\varepsilon)\|^2)
\end{equation}
Thus, the algorithm is globally exponentially convergent with the exponential rate no more than $-\frac{\zeta_3c_0}{\zeta_1c_0 +\zeta_2 + \zeta_3}$, which implies the conclusion.
\end{proof}

\begin{remark}
The obtained result about the exponential rate is consistent with some existing ones for undirected graphs such as \cite[Theorem 4.3]{Yi2016Initialization}, but our algorithm is of lower dimensional dynamics and also applicable to balanced directed graphs.
\end{remark}

Next, we need to verify the sub-optimality of the algorithm \eqref{eq:algorithmCompact} and check the difference between the sub-optimal solution and the optimal one.

\begin{theorem}\label{thm:3}
The equilibrium $(\bar{\bm{x}}(\varepsilon), \bar{\bm{\lambda}}(\varepsilon))$ of algorithm \eqref{eq:algorithmCompact} is a sub-optimal solution of problem \eqref{eq:optimizationProblem} in the sense that
\begin{equation}\label{eq:resouceAllocationConstraint}
(1^T_{N}\otimes I_n) \bar{\bm{x}}(\varepsilon) = d,
\end{equation}
and
\begin{equation}\label{eq:solutionLimit}
\lim_{\varepsilon \to 0} (\bar{\bm{x}}(\varepsilon), \bar{\bm{\lambda}}(\varepsilon)) = (\bm{x}^*, \bm{\lambda}^*).
\end{equation}
Moreover, for any $\varepsilon \in (0,\varepsilon_0)$, there hold
\begin{equation}\label{eq:LipchitzContinuous}
\|\bar{\bm{x}}(\varepsilon) -  \bm{x}^*\| \leq \gamma_1 \varepsilon, \quad \|\bar{\bm{\lambda}}(\varepsilon) -  \bm{\lambda}^*\| \leq \gamma_2 \varepsilon,
\end{equation}
where $\gamma_1 \triangleq  k_1(k_1+1)\|\bm{\lambda}_0\|, \gamma_2 \triangleq \gamma_1 \sup_{\|\bm{z}\|\leq \delta}\{\|\nabla^2 f(\bm{z})\|\}$, and $\varepsilon_0, k_1, \bm{\lambda}_0, \delta$ are in the proof of Theorem \ref{thm:1}.
\end{theorem}

\begin{proof}
Since $(\bar{\bm{z}}(\varepsilon), \bar{\bm{\lambda}}(\varepsilon))$ satisfies \eqref{eq:equilibriumCompact} and $(1^T_{N}\otimes I_n)\bm{L} = 0, (1^T_{N}\otimes I_n)\bm{b} = d$, equality \eqref{eq:resouceAllocationConstraint} holds.

Next, from the proof of Theorem \ref{thm:1}, $\bar{\bm{z}}(\varepsilon) = \bar{\bm{x}}(\varepsilon) -  \bm{x}^*$ is a fixed point of map $\Phi(\bm{r}(\bm{z}), \varepsilon)$. It follows from \eqref{eq:rz2} and \eqref{eq:maptoitself} that
\begin{equation}
\begin{aligned}
\|\bar{\bm{x}}(\varepsilon) -  \bm{x}^*\| & = \|\Phi(\bm{r}(\bar{\bm{x}}(\varepsilon) - \bm{x}^*), \varepsilon)\|\\
& \leq k_1 \|\lambda_0\| \varepsilon + k_1\|\bm{r}(\bar{\bm{x}}(\varepsilon) - \bm{x}^*)\| \\
& \leq k_1 \|\lambda_0\| \varepsilon + \frac{k_1}{k_1+1}\|\bar{\bm{x}}(\varepsilon) -  \bm{x}^*\|
\end{aligned}
\end{equation}
Thus there holds
\begin{equation}
\|\bar{\bm{x}}(\varepsilon) -  \bm{x}^*\| \leq k_1(k_1+1)\|\bm{\lambda}_0\|\varepsilon.
\end{equation}
On the other hand, it follows from \eqref{eq:maptoitself} that
\begin{equation}
\|\bar{\bm{x}}(\varepsilon) -  \bm{x}^*\| \leq \delta.
\end{equation}
Thus
\begin{equation}
\begin{aligned}
\|\bar{\bm{\lambda}}(\varepsilon) -  \bm{\lambda}^*\| & = \|\nabla f(\bar{\bm{x}}(\varepsilon)) - \nabla f(\bm{x}^*)\| \\
& \leq \sup_{\|\bm{z}\|\leq \delta}\{\|\nabla^2 f(\bm{z})\|\} \cdot \|\bar{\bm{x}}(\varepsilon) -  \bm{x}^*\|
\end{aligned}
\end{equation}
Therefore, \eqref{eq:LipchitzContinuous} holds, which also implies \eqref{eq:solutionLimit}.
\end{proof}

Note that the sub-optimal solution depends closely on not only the parameter $\varepsilon$, but also the parameter $\bm{b}$. The following result shows a special case when $\bm{b}$ happens to equal the $\bm{x}^*$.

\begin{corollary}
With $\bm{L}\bm{\lambda_0} = \bm{b}- \bm{x}^*$, we can chose $\bm{\lambda}_0 = 0$ provided $\bm{b} = \bm{x}^*$. Then algorithm \eqref{eq:algorithmCompact} with any $\varepsilon >0$ has its equilibrium as $(\bm{x}^*, \bm{\lambda}^*)$, i.e., it gives exactly the optimal solution to problem \eqref{eq:optimizationProblem}.
\end{corollary}

\begin{remark}
Theorems \ref{thm:2} and \ref{thm:3} showed that, different from some algorithms like \eqref{eq:algorithmPI}, the algorithm \eqref{eq:algorithmNew} is of simple dynamics, and convergent over weight-balanced graphs, without depending on the network topology; also, it can be adjusted easily to reduce the optimization error by tuning the parameter $\varepsilon$.
\end{remark}

\subsection{Numerical example}

Here we give an illustrative example for our algorithm. Consider the following problem
\begin{equation*}
\begin{aligned}
& \min f(\bm{x}) = \frac{1}{2}(x_1^2 + \frac{1}{4}x_2^2 + x_3^2)\\
& \text{ s.t.    } x_1 + x_2 + x_3 = 1
\end{aligned}
\end{equation*}
with a multi-agent system consisting of three agents, where agent $i$ manipulates variable $x_i$ for $i = 1, 2, 3$, and their interaction graph is shown in Fig. \ref{fig:topology}.

Our distributed algorithm can be given as follows:
\begin{equation*}
\left\{\begin{aligned}
\dot{x}_1 & = - x_1 -\lambda_1\\
\dot{x}_2 & = - \frac{1}{4}x_2 -\lambda_2\\
\dot{x}_3 & = - x_3 - \lambda_3 \\
\varepsilon \dot{\lambda}_1 & = - (\lambda_1 - \lambda_3) + \varepsilon (x_1 - \frac{1}{3}) \\
\varepsilon \dot{\lambda}_2 & = - (\lambda_2 - \lambda_1) + \varepsilon (x_2 - \frac{1}{3}) \\
\varepsilon \dot{\lambda}_3 & = - (\lambda_3 - \lambda_2) + \varepsilon (x_3 - \frac{1}{3}) \\
\end{aligned}\right.
\end{equation*}
By some calculations, the equilibrium point $(\bar{\bm{x}}(\varepsilon), \bar{\bm{\lambda}}(\varepsilon))$ is
\begin{equation*}
\begin{aligned}
\begin{bmatrix}
\bar{x}_1(\varepsilon)\\
\bar{x}_2(\varepsilon)\\
\bar{x}_3(\varepsilon)
\end{bmatrix} & =
\begin{bmatrix}
\frac{1}{6}\\
\frac{2}{3}\\
\frac{1}{6}
\end{bmatrix} +
\frac{\varepsilon}{6(4\varepsilon^2 + 9\varepsilon + 6)}
\begin{bmatrix}
4\varepsilon + 9\\
- 8\varepsilon - 12\\
4\varepsilon +3
\end{bmatrix} \\
\begin{bmatrix}
\bar{\lambda}_1(\varepsilon)\\
\bar{\lambda}_2(\varepsilon)\\
\bar{\lambda}_3(\varepsilon)
\end{bmatrix} & =
\begin{bmatrix}
-\frac{1}{6}\\
-\frac{1}{6}\\
-\frac{1}{6}
\end{bmatrix} +
\frac{\varepsilon}{6(4\varepsilon^2 + 9\varepsilon + 6)}
\begin{bmatrix}
-(4\varepsilon + 9)\\
2\varepsilon + 3\\
-(4\varepsilon +3)
\end{bmatrix}
\end{aligned}
\end{equation*}
Indeed, the optimal solution of the problem is $\bm{x}^* = (\frac{1}{6}, \frac{2}{3}, \frac{1}{6})^T$ because, with the Cauchy inequality, $(x_1^2 + \frac{1}{4}x_2^2 + x_3^2)(1^2 + 2^2 + 1^2) \geq (x_1+x_2+x_3)^2 = 1$ and equality holds if and only if $\bm{x} = k(1,2,1)^T$ for some $k\in \mathbb{R}$. Due to the equality constraint, $k$ must be $\frac{1}{6}$. Moreover, we observe that $\bm{\bar{x}}(\varepsilon)$ satisfies the constraint, i.e., $\bar{x}_1(\varepsilon) + \bar{x}_2(\varepsilon) + \bar{x}_3(\varepsilon) = 1$. Furthermore, the distance between the algorithm equilibrium and the optimal solution is dominated by a term proportional to $\varepsilon$.

Simulations are taken with $\varepsilon = 1$, $\varepsilon = 0.1$, and $\varepsilon = 0.01$. The trajectories and the Lyapunov function are shown in Fig. \ref{fig:F1} - Fig. \ref{fig:F4}. It is indicated that our simple algorithm converges to its equilibrium, which approaches the optimal point as $\varepsilon$ tends to zero.

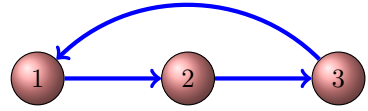
\begin{figure}
\centering
\begin{tikzpicture}[scale = 2]

\tikzstyle{every node}=[shape=circle,draw,minimum size = 20 pt,ball color = red!40]
\path (180 : 1cm)   node (1) {$\,1\,$};
\path (0 : 0cm)  node (2) {$\,2\,$};
\path (0 : 1cm)  node (3) {$\,3\,$};

\draw [->][ultra thick] [blue]  (1) -- (2);
\draw [->][ultra thick] [blue]  (2) -- (3);
\draw [->][ultra thick]  [blue] (3) to [out=135,in=45] (1);

\end{tikzpicture}
\caption{The communication graph of the three agents.
}\label{fig:topology}
\end{figure}


\begin{figure}
\begin{center}
\includegraphics[width=7.4cm]{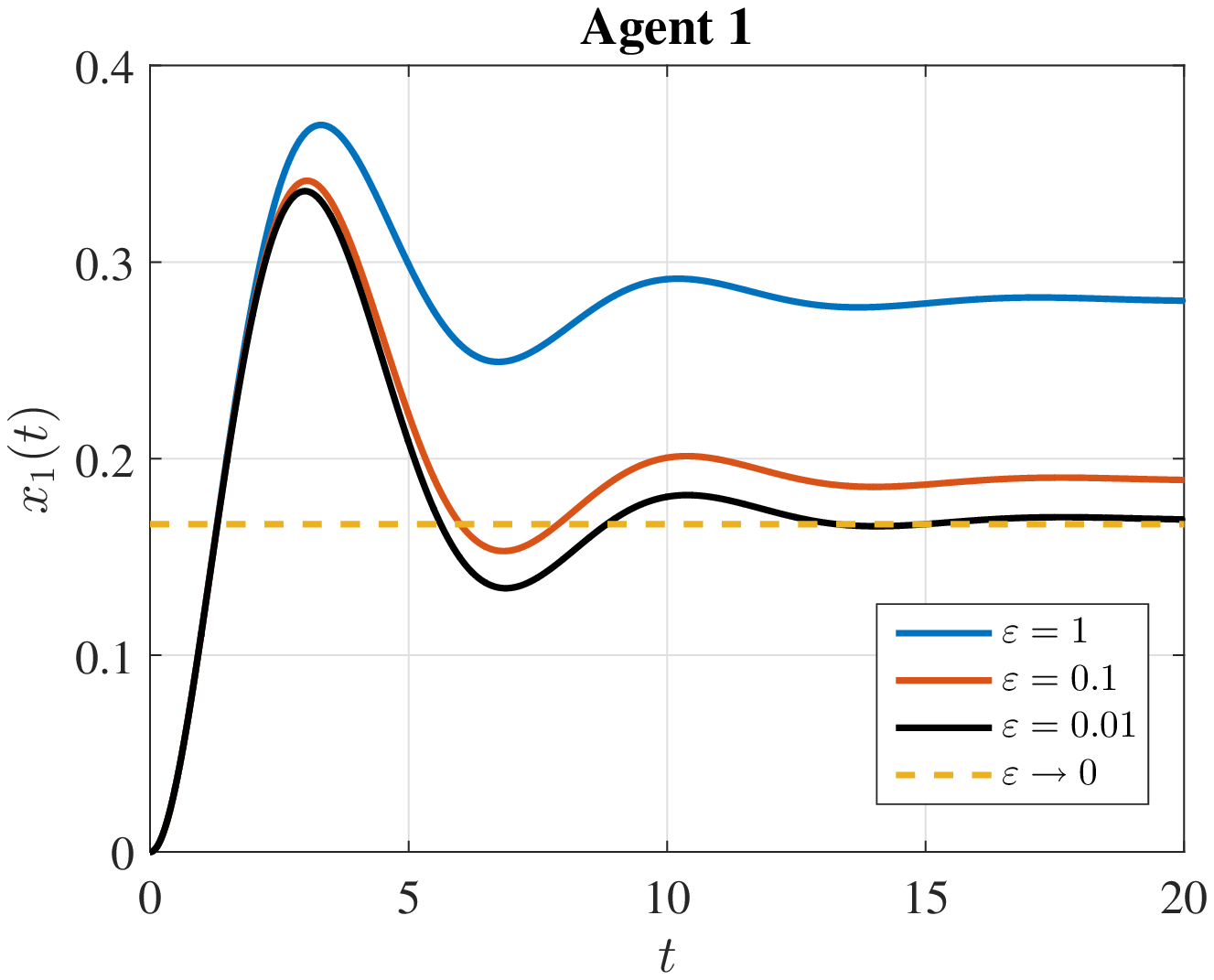}    
\caption{The trajectories of allocation of agent $1$} \label{fig:F1}
\end{center}
\end{figure}

\begin{figure}
\begin{center}
\includegraphics[width=7.4cm]{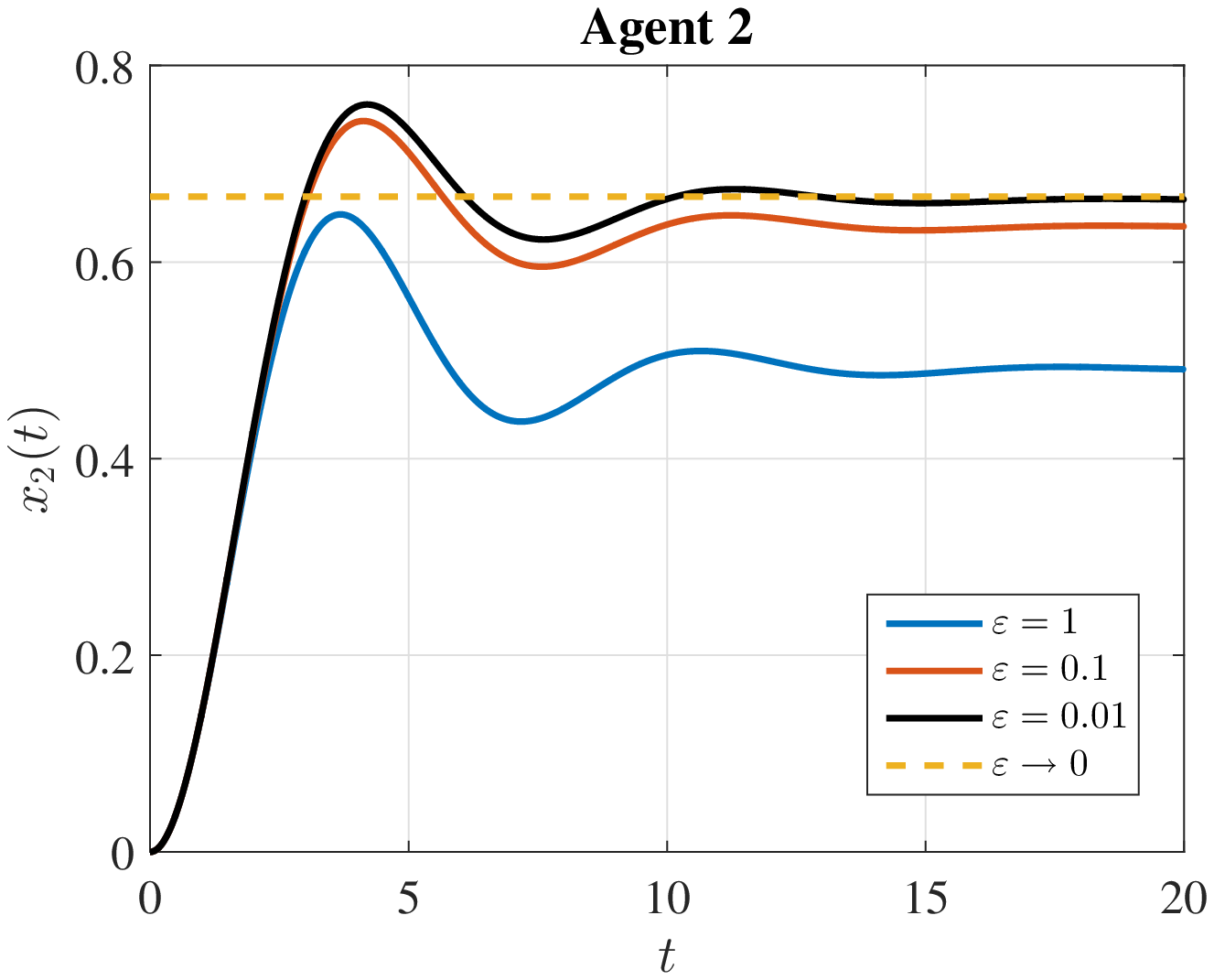}    
\caption{The trajectories of allocation of agent $2$} \label{fig:F2}
\end{center}
\end{figure}

\begin{figure}
\begin{center}
\includegraphics[width=7.4cm]{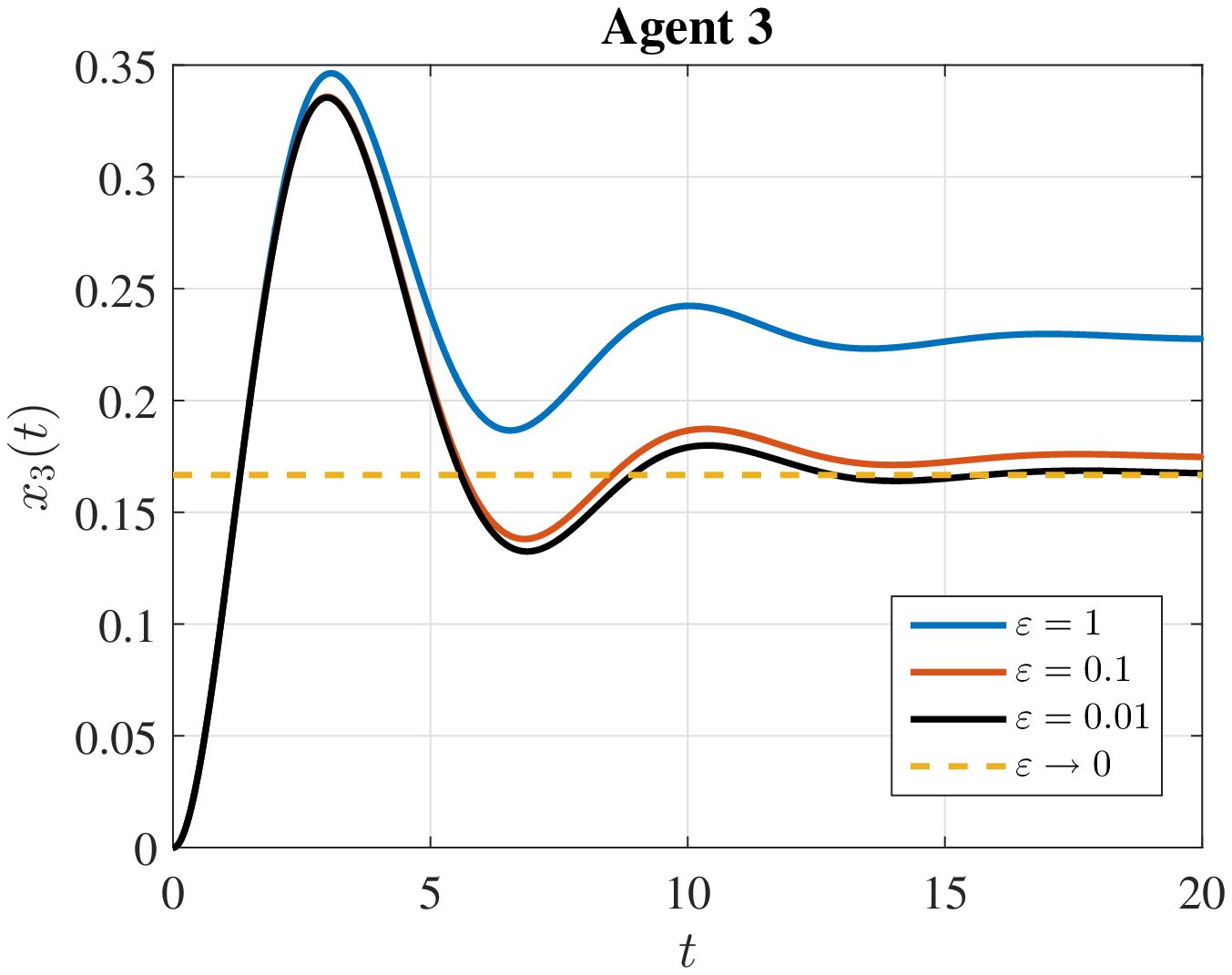}    
\caption{The trajectories of allocation of agent $3$} \label{fig:F3}
\end{center}
\end{figure}

\begin{figure}
\begin{center}
\includegraphics[width=7.4cm]{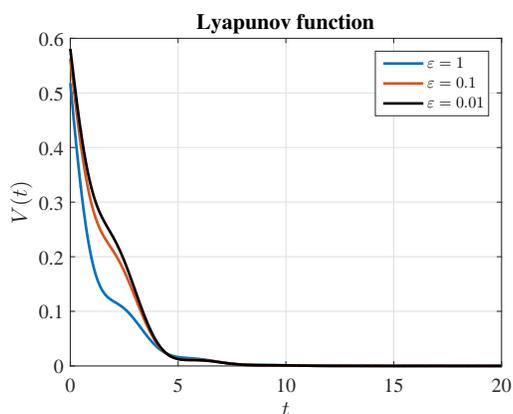}    
\caption{The trajectories of Lyapunov function} \label{fig:F4}
\end{center}
\end{figure}


\section{Conclusions}
In this paper, a distributed sub-optimal continuous-time algorithm has been proposed for resource allocation optimization problem. The convergence has been proved over any strongly connected and weight-balanced graph and the sub-optimality have been analyzed with numerical simulation. At the same time, the singular perturbation ideas have been shown to be useful in the distributed sub-optimal design, though the problems occurred are not completely covered by the existing singular perturbation theory.  In fact, based on the proposed approach, we are considering some systematical ways to further make the singular perturbation techniques serve the distributed algorithm design with various constraints.

\bibliographystyle{dcu}        

\bibliography{E:/HongLab/bib/refference0,E:/HongLab/bib/refference1}

\end{document}